\documentclass[12pt]{article}
\usepackage{graphicx}
\usepackage{amssymb}
\usepackage{epstopdf}
\usepackage{amsmath}
\usepackage{amssymb}
\usepackage{amsthm}
\newtheorem{theorem}{Theorem}[section]
\newtheorem{lemma}[theorem]{Lemma}
\newtheorem{corollary}[theorem]{Corollary}
\newtheorem{proposition}[theorem]{Proposition}

\newtheorem{fact}[theorem]{Fact}

\newtheorem{conjecture}[theorem]{Conjecture}

\title{Kalai's conjecture in $r$-partite $r$-graphs}
\author{Maya Stein\footnote{Department of Mathematical Engineering, University of Chile, and Center for Mathematical Modeling,   UMI 2807 CNRS. Supported by Fondecyt Regular Grant 1183080 and by CONICYT + PIA/Apoyo a centros cient\'ificos y tecnol\'ogicos de excelencia con financiamiento Basal, C\'odigo AFB170001.}\bigskip \\ University of Chile\\ }

\begin{document}
\maketitle

\begin{abstract}
Kalai conjectured that every $n$-vertex  $r$-uniform hypergraph with more than $\frac{t-1}{r} {n \choose r-1}$ edges  contains all tight $r$-trees of some fixed size~$t$. 
We prove Kalai's conjecture for $r$-partite $r$-uniform hypergraphs. 
Our result is asymptotically best possible up to replacing  the term $\frac{t-1}{r}$ with the term $\frac{t-r+1}{r}$.

We apply our main result  in graphs to show an upper bound for the Tur\'an number of trees.
\end{abstract}

\section{Introduction}

For graphs, the well-known Erd\H os-S\'os conjecture from 1963 states that any graph with more than $(t-1)\frac n2$ edges contains,  as subgraphs, all trees with $t$ edges. 
In 1984, Kalai introduced a natural generalisation of this conjecture to uniform hypergraphs.
For simplicity, from now on $r$-uniform hypergraphs will be called {\it $r$-graphs}.

In order to be able to state Kalai's conjecture, we need to clarify the notion of a {\em tree} in an $r$-graph. The definition we will use  relies on the following construction. Start with an $r$-edge $e_1$ and the $r$ vertices it contains: This is~$T_1$. Now, in every step $i$, we may add a new edge~$e_i$. It is required that $e_i$ contains precisely one new vertex $v_i$, and  that $e_i\setminus \{v_i\}$ is a subset of some edge of~$T_{i-1}$. The new hypergraph is $T_i$.  Any hypergraph that can be constructed in this way will be called a {\em tight $r$-tree}.

Here is Kalai's generalisation of the  Erd\H os-S\'os conjecture.

\begin{conjecture}[Kalai 1984, see~\cite{FranklFuredi87}]\label{kalaiC}
Let $H$ be an $r$-graph on $n$ vertices with more than $\frac{t-1}{r} {n \choose r-1}$ edges.
Then $H$ contains every tight $r$-tree $T$ having $t$ edges.
\end{conjecture}

As already  noted in~\cite{FranklFuredi87}, it follows from constructions using a result of R\"odl~\cite{roedlPackCov} (or alternatively, one can use designs whose existence is guaranteed by Keevash's work~\cite{keevashDesigns}) that this conjecture is tight as long as certain divisibility conditions are satisfied. 

It has been observed  (see e.g.~\cite{FurediJiang15}) that if the bound in Conjecture~\ref{kalaiC} is multiplied with a factor of $r$ then the conjecture holds:
\begin{equation}\label{badbound}
\text{Conjecture~\ref{kalaiC} holds if the bound is replaced by  $(t-1) \textstyle\binom{n}{r-1}$.}
\end{equation}
 The reason is that we can successively delete edges from the host $r$-graph until arriving at an $r$-graph $H'$ having the property that each $(r-1)$-subset~$S$ of~$V(H)$ either belongs to $0$ or to at least $t$ edges.  Then we can embed the tree greedily into $H'$, following the given ordering of the edges. 

Not much is known on Kalai's conjecture in general, except for the case $r=2$. We refer to~\cite{tree-survey} for an overview of known results in this case. 

The known results for $r\ge 3$ all focus on specific types of tight $r$-trees. 
Frankl and F\"uredi~\cite{FranklFuredi87} show that Conjecture~\ref{kalaiC} holds for all `star-shaped' tight $r$-trees, that is, tight $r$-trees  whose first edge  intersects each other edge in $r-1$ vertices. 
F\"uredi, Jiang, Kostochka, Mubayi and Verstra\"ete~\cite{FJKMV2019, FJKMV18} show versions of Conjecture~\ref{kalaiC} for a broadened variant of the concept of `star-shaped' (instead of the first edge, there is a constant number of first edges intersecting all other edges), and F\"uredi and Jiang~\cite{FurediJiang15} show the conjecture for special types of tight $r$-trees with many leaves.

For tight $r$-paths, bounds on the number of edges of the host $r$-graph below the bound $(t-1) {n \choose r-1}$ from~\eqref{badbound} were established by Patk\'os~\cite{patkos} and by F\"uredi, Jiang, Kostochka, Mubayi and Verstra\"ete~\cite{FJKMV17}. Namely, the bound in~\eqref{badbound} can be replaced by $\frac{t-1}2 {n \choose r-1}$ if $r$ is even, and by a similar bound if $r$ is odd.
An asymptotic version of Kalai's conjecture for tight $r$-paths whose order is  linear in  the order $n$ of the  host $r$-graph has been confirmed by Allen, B\"ottcher, Cooley and Mycroft~\cite{ABCM17} for large $n$. 

Also, the authors of~\cite{FJKMV2019} show that if one  replaces the bound in Kalai's conjecture with $\frac{t-1}{r}|\partial H|$, an equivalent conjecture is obtained. (As it is usual, we define the {\em shadow} $\partial H$ of an $r$-graph $H$ as the set of all $(r-1)$-sets contained in edges of $H$.)

Our first contribution is a solution of Kalai's  conjecture for $r$-partite $r$-graphs.
We will actually show our result with the bound from~\cite{FJKMV2019}, that is, the term~${n \choose r-1}$ from Kalai's conjecture will be replaced with the smaller term~$|\partial H|$. 

\begin{theorem}\label{r-part}\label{r-partK}
Let $r\ge 2$ and let $H$ be an $r$-partite $r$-graph. If $H$ has more than $\frac{t-1}r|\partial H|$ edges, 
then $H$ contains every tight $r$-tree $T$ having $t$ edges.
\end{theorem}

The proof of Theorem~\ref{r-part} relies on an auxiliary lemma, Lemma~\ref{r-part_mindeg}, which might be interesting in its own right.	
Lemma~\ref{r-part_mindeg} is a variation of the earlier observation  on subhypergraphs of convenient codegree that led to~\eqref{badbound}  (as usual, we will say $S\in\partial H$ has {\em codegree} $c$ if $S$ lies in $c$ edges of $H$).
The novelty in Lemma~\ref{r-part_mindeg}  is that it allows for a different minimum codegree into each of the partition classes. For instance, for bipartite graphs of average degree exceeding~$t_1+t_2-2$, the lemma yields a subgraph  having minimum degree at least~$t_1$ in one direction, and  $t_2$ in the other direction, for any integers $t_1$, $t_2$ (see Corollary~\ref{coro:r-part_mindeg}). For $r$-partite $r$-graphs, Lemma~\ref{r-part_mindeg} gives an analogous statement based on codegrees.

Theorem~\ref{r-part} is not very far from best possible. 
We will see in Proposition~\ref{notbetter} that there are 
$r$-partite $r$-graphs $H$ not containing all  tight $r$-trees with $t$ edges fulfilling 
\begin{equation}\label{Hexa}|E(H)|\sim \frac{t-r+1}r|\partial H|.\end{equation}
 For $r\ge 3$, this might indicate some room for a small improvement of the bound $\frac{t-1}r|\partial H|$ from  Theorem~\ref{r-part}, with other methods than the ones used here. 

 Theorem~\ref{r-partK} also has an interesting application. Namely, 
for graphs, Theorem~\ref{r-partK} can be used to obtain an  upper
 bound for the {\it Tur\'an number} $ex(n,T)$ of a tree $T$ (this is the maximum number of edges a graph on $n$ vertices can have without necessarily containing~$T$ as a subgraph).
 
Observe that for $r=2$, the bound implied by~\eqref{badbound} for the Tur\'an number of a $t$-edge tree $T$  is 
\begin{equation}\label{erdo}ex(n,T)\le (t-1)n,\end{equation}
 which is  a factor of $2$ away from the bound $ex(n,T)\le \frac{(t-1)n}2$ one can cal\-culate from the Erd\H os--S\'os conjecture.
We will see in Proposition~\ref{betterbound} that the  bound~\eqref{erdo}  can be replaced by the slightly better bound 
\begin{equation}\label{turan}
ex(n,T)\le\frac{t}{t+1}(t-1)n.
\end{equation}

Moreover, if $t$ is even, the term $\frac{t}{t+1}$ can be replaced with the term $\frac{t-1}{t}$.
 We achieve the bound~\eqref{turan} by considering a maximum {\it $2$-cut}, that is, a bipartition  of the vertices of a graph $G$ that maximises the number of edges crossing the bipartition. A result of Alon, Krivelevich and Sudakov~\cite{AKS05} states that, if a fixed tree is excluded from $G$, then one can guarantee that substantially more than half of the edges of $G$ cross some  $2$-cut. We then apply Theorem~\ref{r-partK} to the bipartite graph spanned by the edges in the cut.

The paper is organised as follows. 
We will state and prove Lemma~\ref{r-part_mindeg} and use it to prove  Theorem~\ref{r-part} in Section~\ref{sec:proof}. In Section~\ref{sec:example} we will prove that Theorem~\ref{r-part} is not far from best possible in the above described sense, at least for balanced trees, by exhibiting $r$-graphs $H$ fulfilling~\eqref{Hexa}. Finally, in Section~\ref{sec:betterbound}, we will prove~\eqref{turan}, our Tur\'an number bound for trees, in Proposition~\ref{betterbound}.

\section{Proof of Theorem~\ref{r-partK}}\label{sec:proof}

The main ingredient for the proof of Theorem~\ref{r-part} is Lemma~\ref{r-part_mindeg} below.
For convenience, let us give a quick definition before we state the lemma. 

If $H$ is an $r$-partite $r$-graph, we say that $\delta_{(1,2,\ldots,r)} (H)\ge (t_1, t_2,\ldots,t_r)$ 
if there is a way of labeling the partition classes of $H$ as $V_1, V_2, \ldots,V_r$ such that for each $i\in[r]$,
every  $S\in \partial H$ missing $V_i$ is contained in at least $t_{i}$ edges of~$H$.

\begin{lemma}\label{r-part_mindeg}
Let $r,  t, t_1, t_2, \ldots, t_r\in\mathbb N$ such that $$t_1 + t_2+ \ldots + t_r=t+r-1,$$ and let $H$ be an $r$-partite $r$-graph  with more than $\frac{t-1}r |\partial H|$ edges.
Then there is a non-empty $r$-graph $H'\subseteq H$ such  that $$
\delta_{(1,2,\ldots,r)} (H')\ge (t_1, t_2,\ldots,t_r).$$
\end{lemma}

\filbreak
Lemma~\ref{r-part_mindeg} has the following corollary.

\begin{corollary}\label{coro:r-part_mindeg}
For all $t_1, t_2\in\mathbb N$ every bipartite graph $G$ with $d(G)>t_1+t_2-2$ has  a non-empty  subgraph $G'=(V_1,V_2)$  such that  each vertex in $V_i$ has degree at least $t_i$ in $G'$, for $i=1,2$.
\end{corollary}

Before proving Lemma~\ref{r-part_mindeg}, let us show how it implies Theorem~\ref{r-part}. For this, we will need the  following fact which is immediate from the definition of tight $r$-trees.

\begin{fact}\label{tunique}
Every tight $r$-tree has a unique $r$-partition.
\end{fact}

Now we are ready to prove our main result.

\begin{proof}[Proof of Theorem~\ref{r-part}]
Assume we are given an $r$-graph $H$, and a tight $r$-tree $T$ with $t$ edges.
Consider the $r$-partition of~$T$ given by Fact~\ref{tunique}, and let $t_1, t_2, \ldots, t_r$ be the sizes of the partition classes. Thus $$t_1 + t_2+ \ldots + t_r=t+r-1.$$ 

We apply Lemma~\ref{r-part_mindeg} to see that there is an $r$-graph $H'\subseteq H$ with $r$-partition $V_1\cup V_2\cup \ldots\cup V_r$, such that for each $i\le r$, any element of $\partial H'$ which avoids $V_i$ is contained in at least~$t_i$ edges of $H'$ (each containing a different vertex from~$V_i$). 
We may therefore embed $T$ following its natural order $v_1, v_2, \ldots, v_{t+r-1}$. At every step $j$, there is an unoccupied vertex we can choose as the image of~$v_j$, because the total number of vertices from $V(T)$ we need to embed in any fixed class $V_i$ is at most $t_i$.
\end{proof}

It only remains to prove Lemma~\ref{r-part_mindeg}.

\begin{proof}[Proof of Lemma~\ref{r-part_mindeg}]
We may assume
that $t_1\le t_2\le \ldots \le t_r$. Set $$\delta_i:= t_i-\frac{t+r-1}r,$$ for all $i=1,\ldots, r$. Clearly, we have
\begin{equation}\label{delta_r}
\delta_1\le \delta_2\le \ldots \le \delta_r\text{ \ and \ }\delta_1\le 0,
\end{equation} 
 and moreover,
\begin{equation}\label{delta}
\sum_{i=1}^r\delta_i= 0.
\end{equation} 

We now turn to the $r$-graph  $H$, with its $r$-partition $V_1\cup V_2\cup \ldots\cup V_r$. We let $h_i$ denote the number of elements of $\partial H$ that avoid $V_i$, for each $i\in[r]$.
Clearly, 
\begin{equation}\label{h_i}
\text{$\sum_{i=1}^r h_i=|\partial H|$,}
\end{equation} 
 and
after possibly relabeling the partition classes of $H$, we may assume that    
\begin{equation}\label{delta-h}
\text{$h_1\ge h_2\ge \ldots \ge h_r$.}
\end{equation}

Let $E_0$ denote the set of all edges of $H$.
For $j\ge 1$, we inductively define the set $E_j$ as follows. If there is an $(r-1)$-set $S\subseteq V(H)$ missing~$V_i$ and contained in at least one, but less than $t_{i}$ edges from $E_{j-1}$, then we set  $E_j:=\{e\in E_{j-1} : e\not\supseteq S\}$.
If there is no set $S\subseteq V(H)$ as above, we terminate the process, and set $E:=E_{j-1}$.

Observe that every $(r-1)$-subset $S$ of $V(H)$ appears in at most one of the steps $j$ as the reason for deleting edges, and in that step, we deleted at most $t_{i}-1$ edges, where $i$ is such that $S$ misses $V_i$. Also, $(r-1)$-sets $S\notin \partial H$ never appear.

Therefore, 
$$|E(H)|  \ \le \  |E| + \sum_{i=1}^r  (t_i-1) h_i.$$

We claim that 
\begin{equation}\label{E}
E\neq\emptyset.
\end{equation} 
Then, we can take $H'$ to be the subhypergraph induced by the edges in $E$, and are done. So it only remains to prove~\eqref{E}.

 In order to see~\eqref{E}, note that otherwise, by our assumption on the number of edges of $H$, we have that
\begin{align*}
\frac{t-1}r|\partial H|\ < \ |E(H)|  \ \le \  \sum_{i=1}^r  (t_i-1) h_i
\ &\le \  \sum_{i=1}^r ( \frac{t+r-1}r-1+\delta_i) h_i \\
\ &\le \  \frac{t-1}r\cdot \sum_{i=1}^r  h_i +\sum_{i=1}^r \delta_i  h_i ,
\end{align*}
and so,  using~\eqref{h_i}, we obtain that
\begin{equation}\label{zero}
\sum_{i=1}^r \delta_i  h_i \ > \ 0.
\end{equation} 
By~\eqref{delta_r}, we can choose an index~$i\in[r]$ such that $\delta_i\le 0$  for all $i\le i^*$ and $\delta_i> 0$ for all $i>i^*$. 
 This  choice of $i^*$, together with~\eqref{delta-h} and~\eqref{delta},  enables us to calculate that
\begin{align*}
\sum_{i=1}^r \delta_i  h_i \  
=\ \sum_{i=1}^{i^*}\delta_i  h_i  + \sum_{i=i^*+1}^r \delta_i  h_i \ 
& \le \    \sum_{i=1}^{i^*} \delta_i h_{i^*} +  \sum_{i=i^*+1}^r \delta_i h_{i^*}\\ 
& \le \   h_{i^*} \cdot \sum_{i=1}^{r} \delta_i  \\ &=0,
\end{align*}
a contradiction to \eqref{zero}. This proves~\eqref{E}, thus completing the proof of the lemma.
\end{proof}

\section{Lower bounds for $r$-partite $r$-graphs}\label{sec:example}

The bounds from Kalai's conjecture cannot be weakened much in $r$-partite $r$-graphs.  This is asymptotically shown in Proposition~\ref{notbetter} below.
However, in this proposition, the term $\frac{t-1}r$ from Kalai's conjecture is replaced with the term $\frac{t-r+1}r$, which for $r\ge 3$ leaves us with a small gap. Possible finer scale improvements are discussed at the end of this section.
 
We call a tight $r$-tree   {\it balanced} if all its partition classes have the same size. 

\begin{proposition}\label{notbetter}
For all $r\ge 2$, $t\ge 1$ such that $t+1$ is a multiple of~$r$, and for all $\varepsilon>0$, there exists an $r$-graph $H$ with  $r$-partition $V(H)=V_1\cup V_2\cup \ldots\cup V_r$  fulfilling
\begin{equation*}\label{extre}
|E(H)|\ge\big(1-\varepsilon\big)\frac{t-r+1}r\cdot \sum_{i=1}^r\prod_{\ell\neq i}|V_\ell|
\end{equation*}
such that $H$  does not contain any balanced tight $r$-tree $T$ with $t$ edges.
\end{proposition}
\begin{proof}
 Consider the  sets $V_i^j$  for ${i\in[r]}$ and $j=1,2$, where $|V_i^1|=\frac{t+1}r-1$  and $|V_i^2|=\gamma^{-1} (\frac{t+1}r-1)$, for $\gamma:=(1-\varepsilon)^{\frac{1}{1-r}}-1$. Let $H$ be the $r$-partite $r$-graph with partition sets $V_i:=V_i^1\cup V_i^2$ (for $i\in[r]$) and all edges $\{v_1, \ldots ,v_r\}$ having the property that $v_{i^*}\in V_{i^*}^1$ for exactly one index $i^*\in[r]$, and $v_i\in V_i^2$ for all $i\neq i^*$.

  It is easy to see that no $r$-tight tree may contain vertices from both $V_i^1$ and $V_i^2$, for any $i \in[r]$. So,
 since $|V_i^1|<\frac{t+1}r$ for all $i$, we see that
 $H$ does not contain any balanced tight $r$-tree $T$ with $t$ edges. 
 
 The number of edges of $H$ is 
\begin{align*}
|E(H)|= \sum_{i=1}^r |V^1_i|\prod_{\ell\neq i} |V^2_\ell| & = r\cdot \gamma^{1-r}\big(\frac{t+1}r-1\big)^r \\ & =
(1+\gamma)^{1-r}\big(\frac{t+1}r-1\big)\cdot \sum_{i=1}^r\prod_{\ell\neq i}|V_\ell|,
\end{align*}
giving the desired bound.
\end{proof}


%
%
%

 In the example behind Proposition~\ref{notbetter}, the host $r$-graph is much larger than the $r$-tree we are looking for, and another error term hides behind the~$\varepsilon$.
On a finer scale, more improvements on Kalai's bounds might be possible for $r$-partite $r$-graphs.
 For $r=2$, 
Gy\'arf\'as, Rousseau and Schelp~\cite{GRS84} determine the extremal number of $t$-edge paths $P_t$ in bipartite graphs with partition classes of sizes $n\ge m$ (that is, the maximum number $ex(n,m;P_t)$ of edges such a bipartite graph can have without necessarily containing $P_t$). In particular, if $t\le m+1$ is odd they obtain
\begin{equation}\label{GRS}
ex(n,m;P_t)=\frac{t-1}{2}(n+m-t+1).
\end{equation}
 Yuan and Zhang~\cite{YZ} conjecture similar results  as~\eqref{GRS} hold for all trees $T$ (the exact bounds depend on how the bipartition sizes of $T$ relate to $n$ and~$m$), and  establish several special cases. 
 
 Analogous improvements might be possible for hypergraphs.
Note that the quantity from~\eqref{GRS} coincides with the number of edges of the $r$-graph from the proof of  Proposition~\ref{notbetter}, for the case $r=2$. For $r= 3$ one might add in all edges meeting all sets $V_i^1$, and the obtained $r$-graph still does not contain $T$. More generally, for $r\ge 4$ one might add in all edges meeting $V_i^1$ in an odd number of indices $i$.

\section{A better Tur\'an bound for all $2$-graphs}\label{sec:betterbound}

We now discuss an implication of  Theorem~\ref{r-part} for tree containment in graphs (i.e.~$2$-graphs) that are not necessarily bipartite. This will establish the bound~\eqref{betterbound} mentioned in the introduction.

We need some easy definitions first. We call a partition of the vertices of a graph into two sets  a {\it $2$-cut}. The {\it size} of a $2$-cut is the number of edges that cross the cut (where an edge is said to {\it cross} the cut if it has one endvertex on either side).

It is well known and easy to prove that every $m$-edge graph has a $2$-cut of size at least $\frac m2$. A random partition achieves this bound in expectation. But this is not best possible. 
 A classical result of Edwards~\cite{edwards} states that instead of only $\frac{m}2$ edges we can actually guarantee a $2$-cut of size at least $\frac{m}2+\Omega(\sqrt{m})$. Even better bounds can be achieved by excluding a fixed subgraph from $G$. Alon, Krivelevich and Sudakov~\cite{AKS05} show that if we exclude any fixed tree~$T$ from the graph $G$, the maximum number of edges crossing some $2$-cut can be bounded as follows.
 
\begin{theorem}[Alon, Krivelevich and Sudakov~\cite{AKS05}]\label{thm:aks}
Let $t>1$ and let $T$ be a $t$-edge tree. Let $G$ be a graph with $m$ edges that does not contain $T$. \\ Then $G$ has a $2$-cut of size at least $\frac m2+\frac m{2t}$ if $t$ is odd, and  of size at least $\frac m2+\frac m{2t-2}$ if $t$ is even.
\end{theorem}

We can use Theorem~\ref{thm:aks} to improve the bound  from~\eqref{badbound}. The following proposition proves the bound~\eqref{turan} we mentioned in the introduction.

\begin{proposition}\label{betterbound}
Let $t\in\mathbb N$ and let $G$ be a graph on $n$ vertices with more than $(1-\frac{1}{t+1})(t-1)n$ edges if $t$ is odd, and with more than $(1-\frac{1}{t})(t-1)n$ edges if $t$ is even.
Then $G$ contains every tree $T$ having $t$ edges.
\end{proposition}

\begin{proof}
We may assume that $t>1$. We only treat the case when $t$ is odd, as the other case is very similar. 

Given $G$ and $T$, we use Theorem~\ref{thm:aks} to either find a copy of $T$ in $G$, or to obtain a $2$-cut of $G$ having size greater than $$\frac{\frac{t}{t+1}(t-1)n}{2}+\frac{\frac{t}{t+1}(t-1)n}{2t}= \frac{t-1}2n.$$ In the latter case, we apply Theorem~\ref{r-part} to the graph induced by this $2$-cut to see that it contains $T$.
\end{proof}

Variants of Edwards' results for hypergraphs have been studied by
Erd\H os and Kleitman~\cite{EK}. They showed that in an $m$-edge $r$-graph, the expected size of an $r$-cut is $\frac{r!}{r^r}m$ (where an {\it $r$-cut} is a partition of the vertices into $r$ sets, edges {\it cross} the cut if they have one vertex in each partition set, and the {\it size} of an $r$-cut is the number of crossing edges).
Recently,  Conlon, Fox, Kwan and Sudakov~\cite{CFKS19}  improved this bound. In particular, they obtain that for $r\ge 3$, every $m$-edge $r$-graph has an $r$-cut of size at least $\frac{r!}{r^r}m+\Omega(m^{\frac 59})$. They conjecture the exponent in the second term can be improved to  $\frac 23$.
Unfortunately, we are not aware of any result in the spirit of Theorem~\ref{thm:aks} for hypergraphs 
and the bound from~\cite{CFKS19} alone does not seem to suffice to prove a meaningful version of Proposition~\ref{betterbound} for $r$-graphs with $r\ge 3$.

\bibliographystyle{acm}
\bibliography{trees}

\end{document}